\documentclass[12pt]{amsart}
\usepackage{amssymb,amsmath,amsthm}
\usepackage{latexsym}
\usepackage{epsfig}
\usepackage{graphicx}

\newtheorem{theorem}{Theorem}[section]

\newtheorem{definition}[theorem]{Definition}

\newtheorem{proposition}[theorem]{Proposition}
\newtheorem{corollary}[theorem]{Corollary}

\newtheorem{twierdzenie}[theorem]{Twierdzenie}
\newtheorem{stwierdzenie}[theorem]{Stwierdzenie}

\newtheorem{wniosek}[theorem]{Wniosek}

\newtheorem{cwiczenie}[theorem]{\'Cwiczenie}

\long\def\symbolfootnote[#1]#2{\begingroup%
\def\thefootnote{\fnsymbol{footnote}}\footnote[#1]{#2}\endgroup}

\newcommand\Z{{\mathbb Z}}

\begin{document}

\title{\bf ${\bf \it q}$-polynomial invariant of rooted trees}
\author{J\'ozef H. Przytycki}
\thispagestyle{empty}

\begin{abstract}
We describe in this note a new invariant of rooted trees. We argue that the invariant is interesting on it own, and that it has 
connections to knot theory and homological algebra. However, the real reason that we propose this invariant to readers  
 is that we deal here with an elementary, interesting, new mathematics, and after reading this essay 
readers can take part in developing the topic, inventing new results and connections to other disciplines of mathematics,  
and likely, statistical mechanics, and combinatorial biology. We also provide a (free) translation of the paper in Polish.

\end{abstract}

\maketitle
\markboth{\hfil{\sc  $q$-polynomial of rooted trees}\hfil}
\ \
\tableofcontents

\section{Introduction}\label{Section 1.1}

\subsection{Quantum plane}\

We know well the Newton binomial formula:
$$(x+y)^n= \sum_{i=0}^n {n \choose i} x^iy^{n-i}.$$
Of course we assume here that the variables commute, that is $yx=xy$.

We can also visualize formulas for coefficients of $(x+y)^n$ by considering the tree $T_{b,a}$ of two long branches of length $b$ and $a$ 
respectively, coming from the root (Figure 1.1).
We can ask now how many different ways there are  to ``pluck" the tree  $T_{b,a}$ one leaf at a time. Of course we can tear (like in the child's play 
``likes not like") the left or right branch. In total: the left should be torn $b$ times and the right $a$ times so the answer about the number of 
different pluckings is clearly ${a+b \choose a}$. Our notes present a dramatic generalization of this example.

\ \\
\centerline{\psfig{figure=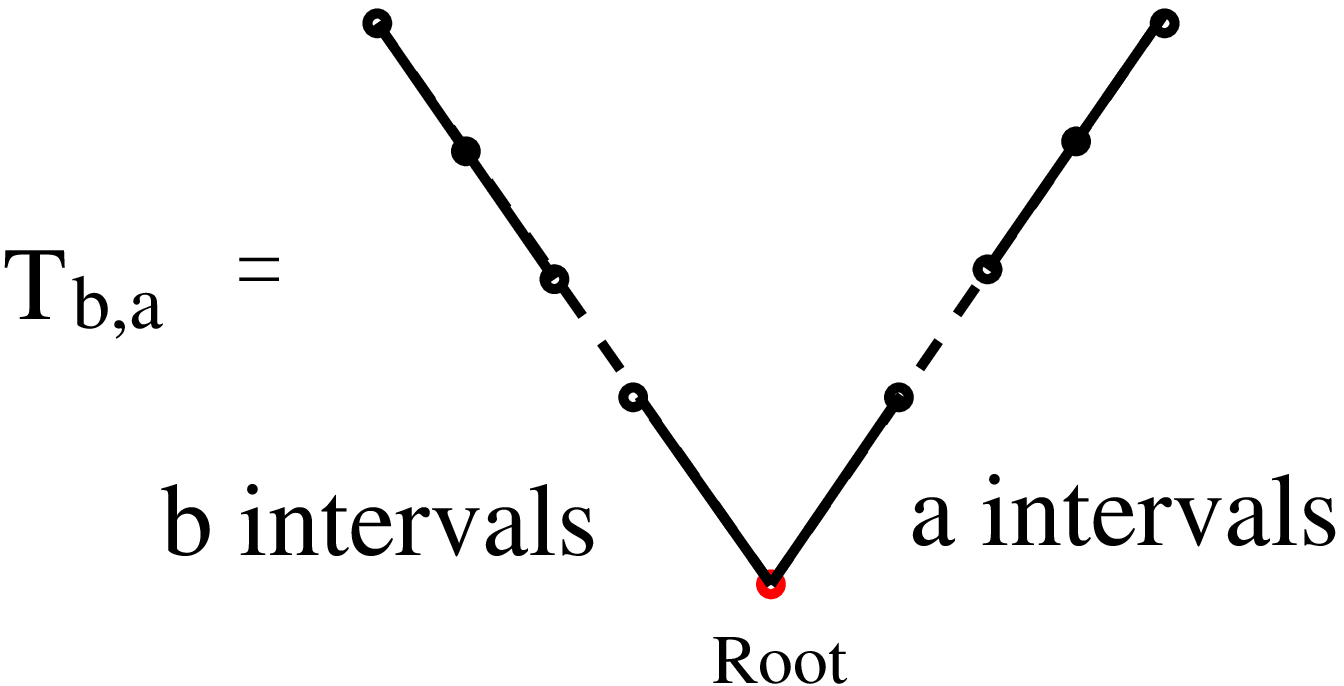,height=3.5cm}}\ \\
\centerline{Figure 1.1; the tree $T_{b,a}$ with branches of length $b$ and $a$}\ \\
\ \\

We can ask now what happens with the binomial formula if we weaken commutativity by replacing it with $yx=qxy$ ($q$ commutes with $x$ and $y$).
In applications in physics, $q$ is often taken to be a complex number but it is better to work generally with the ring $Z[q]$ and 
q-commutative\footnote{By $q$-commutative we understand exactly $yx=qxy$.} polynomials of variables $x$ and $y$ over this ring 
(we call this ring of polynomials the quantum plane or noncommutative plane).
The formula for $(x+y)^n$ in the quantum plane was known already in the XIX century, but if we do not know the result it is good to first 
work out small examples:
$(x+y)^2= y^2+ xy +yx + x^2= y^2 + (1+q)xy +x^2,$ \\
$(x+y)^3= y^3+ xy^2+ yxy +y^2x + x^2y+ xyx + yx^2 +x^3 =$ \\
$y^3 + (1+q+q^2)xy^2 + (1+q+q^2)x^2y+ y^3,$\\
$(x+y)^4= y^4 + xy^3+ yxy^2 +y^2xy + y^3x + x^2y^2 + xyxy + xy^2x + yx^2y + yxyx + y^2x^2 +
          x^3y+ x^2yx + xyx^2 + yx^3 + x^4 =$
$ y^4 +  (1+q+q^2+q^3)xy^3 + (1+ q + 2q^2+ q^3 + q^4)x^2y^2 + (1+q+q^2+q^3)x^3y +x^4 =$
         $ y^4+(1+q+q^2+q^3)xy^3 + (1+q^2)(1+q+q^2)x^2y^2+ (1+q+q^2+q^3)x^3y +x^4$.\\
Observe that we can think of  $1+q+...+q^{n-1}$ as a $q$-analogue of the number $n$.
We define formally $[n]_q=1+q+...+q^{n-1}$. In particular, for $q=1$, $[n]_q=n$.
In the same vein we define $q$ factorial of the $q$-number $[n]_q$ as: $[n]_q!=[n]_q[n-1]_q\cdots [2]_q[1]_q$, and
$q$-analogue of the binomial coefficient\footnote{Called also the Gaussian polynomial, $q$-polynomial of Gauss, Gaussian binomial coefficient, Gaussian coefficient, 
or $q$-binomial coefficient.} by  $\binom{n}{i}_q = \frac{[n]_q!}{[i]_q! [n-i]_q!}$.
Sometimes to stress symmetry, $i \leftrightarrow n-i$, of the Gaussian polynomial we write  
$\binom{n}{i,n-i}_q$. We also observe that the coefficient of $x^2y^2$ in 
$(x+y)^4$ is $(1+q^2)(1+q+q^2)= \frac{1+q+q^2+q^3}{1+q}(1+q+q^2)=\frac{[4]_q[3]_q}{[2]_q}= \binom{4}{2}_q$.

In the notation we introduced, our calculations can be concisely written as:
$$(x+y)^2= y^2 +[2]_qxy +x^2,$$
$$(x+y)^3= y^3 + [3]_qxy^2 + [3]_qx^2y+x^3,$$
$$(x+y)^4= y^4+ [4]_qxy^3 + \binom{4}{2}_qx^2y^2 + [4]_qx^3y+ x^4.$$
Now it is not difficult to guess the general formula for $(x+y)^n$ in the quantum plane:
\begin{proposition}\label{Proposition 1.1} If $yx=qxy$ then
$$(x+y)^n= \sum_{i=0}^n {n \choose i}_qx^iy^{n-i}.$$
\end{proposition}
\begin{proof}
The simplest proof is by induction on $n$ however one can also find proofs without words, interpreting 
combinatorially left and right sides of the equation.\\
{\bf Inductive proof:} We already checked the formula for $n\leq 4$, and we should add that we need the convention (as in the classical case).\\ 
that $[0]_q!=1$ and consequently ${n \choose 0}_q=1 ={n \choose n}_q$. Now we perform the inductive step (from $n-1$ to $n$):
$$(x+y)^n= (x+y)(x+y)^{n-1}=(x+y)\sum_{i=0}^{n-1} {n-1 \choose i,n-i-1}_qx^iy^{n-i-1}=$$
$$\sum_{i=0}^{n-1} {n-1 \choose i,n-i-1}_qx^{i+1}y^{n-i-1} +\sum_{i=0}^{n-1}q^i{n-1 \choose i,n-i-1}_qx^iy^{n-i} =$$
$$\sum_{i=0}^n({n-1 \choose i-1,n-i}_q +q^i{n-1 \choose i,n-i-1}_q)x^iy^{n-i}.$$
We use here the convention that ${n-1 \choose -1,n}_q=0$. \\
We are left now to check that $${n-1 \choose i-1,n-i}_q +q^i{n-1 \choose i,n-i-1}_q)={n \choose i}_q.$$ 
We encourage the reader to check it by themselves before looking at the following calculation.
\begin{enumerate}
\item[(i)] $[a+b]_q= 1+q+...+q^{a+b-1}= (1+q+...+q^{a-1}) + q^a(1+q+...+q^{b-1})= [a]_q+q^a[b]_q= [b]_q+q^b[a]_q$.
\item[(ii)] $$\binom{a+b}{a,b}_q=\frac{[a+b]_q!}{[a]_q! [b]_q!}=[a+b]_q\frac{[a+b-1]_q!}{[a]_q! [b]_q!}=$$
$$([a]_q+q^a[b]_q)\frac{[a+b-1]_q!}{[a]_q! [b]_q!}= \frac{[a+b-1]_q!}{[a-1]_q! [b]_q!} + q^a \frac{[a+b-1]_q!}{[a]_q! [b-1]_q!}= $$
$$\binom{a+b-1}{a-1,b}_q + q^a\binom{a+b-1}{a,b-1}_q= \binom{a+b-1}{a,b-1}_q + q^b\binom{a+b-1}{a-1,b}_q.$$
\end{enumerate} 
\end{proof}

We can now go back to the tree $T_{b,a}$ and play the game in the $q$-fashion that is the leaf taken from the right ($a$) branch is 
counted as $1$ but the leaf taken from the left is counted with the weight $q^a$. Thus we eventually get not the plucking number but 
the plucking polynomial $Q(T_{b,a})$ which by definition satisfies the recursive relation $Q(T_{b,a})= Q(T_{b,a-1}) + q^aQ(T_{b-1,a})$. 
We also notice that $Q(T_{0,a})=Q(T_{b,0})=1$, thus we immediately recognize that the plucking polynomial is $Q(T_{b,a})= \binom{a+b}{a,b}_q$. 
This is the starting point to our definition of the polynomial of plane rooted trees mentioned in the title of the note.

We can repeat our considerations with many variables, $x_1,x_2,...,x_k$. If variables commute we get the familiar multinomial formula
(e.g. familiar to every student taking multivariable calculus):
$$(x_1+x_2+...+x_k)^n= \sum_{a_1,...,a_k; \sum a_i=n}^n {n \choose a_1,...,a_k}x_1^{a_1} x_2^{a_2}\cdots x_k^{a_k},$$
where ${n \choose a_1,...,a_k} = \frac{[n]!}{[a_1]!...[a_k]!}$.
As before the pleasant interpretation of ${n \choose a_1,...,a_k}$ is the number of pluckings of a tree $T_{a_k,...,a_2,a_1}$ of $k$ long 
branches of length $a_k,...,a_2,a_1$, respectively, as illustrated in Figure 1.2.

\ \\
\centerline{\psfig{figure=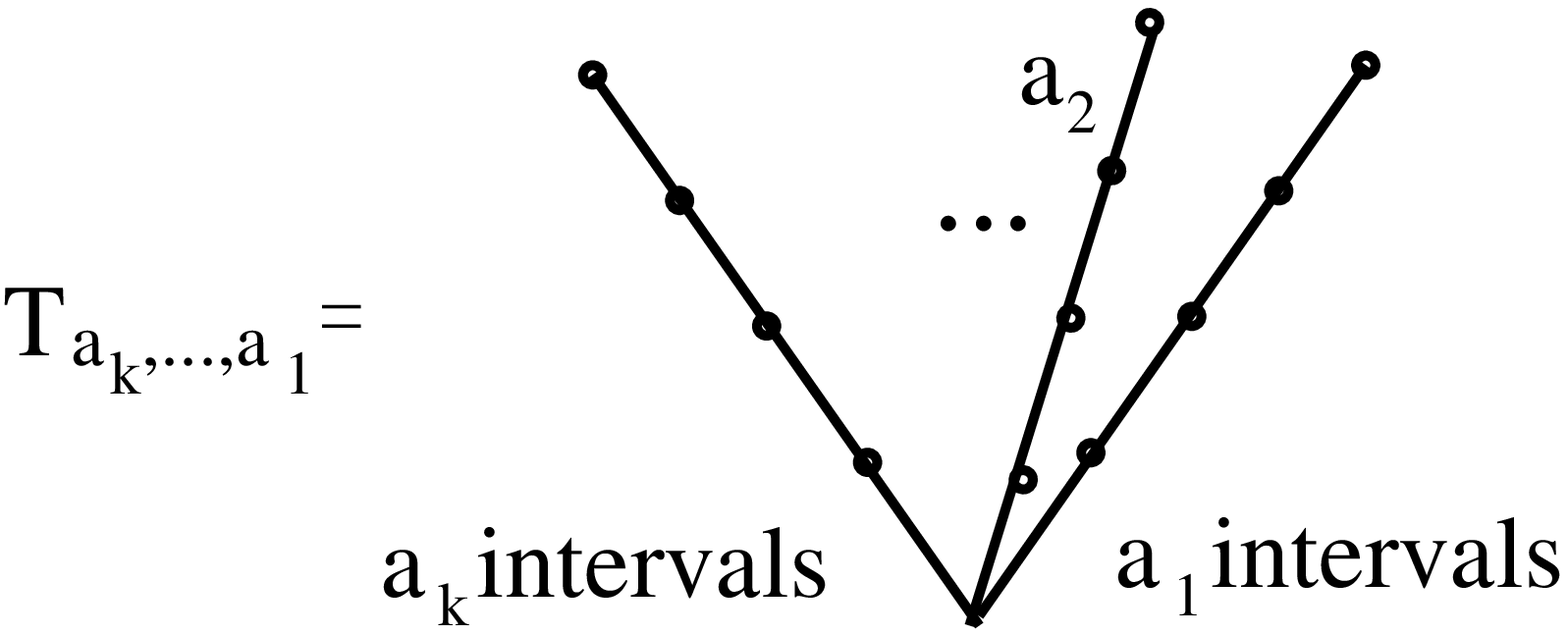,height=4.5cm}}\ \\
\centerline{Figure 1.2; the tree $T_{a_k,...,a_1}$ with branches of length $a_k,...,a_1$}\ \\
\ \\

We can now consider a noncommutative space with $x_jx_i=qx_ix_j$, for $i<j$. We will get the q-multinomial formula:
$$(x_1+x_2+...+x_k)^n= \sum_{a_1,...,a_k; \sum a_i=n}^n {n \choose a_1,...,a_k}_qx_1^{a_1} x_2^{a_2}\cdots x_k^{a_k},$$
where ${n \choose a_1,...,a_k}_q = \frac{[n]_q!}{[a_1]_q!...[a_k]_q!}$.
To prove the formula we can again use an induction on $n$ and again the following two properties are of value:
\begin{enumerate}
\item[(i)] $[a_1+a_2+...+a_k]_q= [a_1]_q + q^{a_1}[a_2]_q + q^{a_1+a_2}[a_2]_q+...+q^{a_1+a_2+...+a_{k-1}}[a_k]_q.$
\item[(ii)] $${a_1+a_2+...+a_k \choose a_1,a_2,...,a_k}_q = 
{a_1+a_2+...+a_k-1 \choose a_1-1,a_2,...,a_k}_q + $$
$$ q^{a_1}{a_1+a_2+...+a_k-1 \choose a_1,a_2-1,...,a_k}_q+...
+ q^{a_1+a_2+...+a_{k-1}} {a_1+a_2+...+a_k-1 \choose a_1,a_2,...,a_k-1}_q.$$
\end{enumerate} 
 
As in the $q$-binomial case, we can interpret the $q$-multinomial coefficients by $q$-plucking the tree   $T_{a_k,...,a_2,a_1}$, that is 
assuming the following plucking formula 
$$Q(T_{a_k,...,a_2,a_1}) = Q(T_{a_k,...,a_2,a_1-1}) + q^{a_1}Q(T_{a_k,...,a_2-1,a_1}) +$$
$$ q^{a_1+a_2}Q(T_{a_k,...,a_3-1,a_2,a_1}) + ... + q^{a_1+a_2+...+a_{k-1}}Q(T_{a_k-1,...,a_2,a_1}).$$  
The recursion is the same as for the $q$-multinomial coefficient so we conclude that 
$$Q(T_{a_k,...,a_2,a_1}) = {a_1+a_2+...+a_k \choose a_1,a_2,...,a_k}_q.$$ 

One can find more properties of $q$-binomial coefficients in \cite{K-C}.

\section{Recursive definition of $q$-polynomial of plane rooted tree}\

Our description of the quantum plane and noncommutative space is already over one hundred years old (e.g. MacMahon (1854-1929)). 
It may, therefore, look surprising that a simple generalization to rooted trees, which we present here, is totally new.
Maybe the explanation lies in the observation that, because there are so many $q$-analogues, a new one is studied only if there 
is an outside reason to do so (like the Jones revolution in knot theory, in my case). Perhaps our $q$-polynomial is buried somewhere 
in a work of MacMahon contemporaries?

We start our definition from the polynomial of a plane (that is embedded in the plane) rooted tree. 
In Corollary \ref{Corollary 2.3} (iii) 
 we show that the polynomial does not depend on the plane embedding and is therefore an invariant of rooted trees. 

In our work we use the convention that trees are growing up (like in Figure 2.1).   
\begin{definition}\label{definition 2.1}
Consider the plane rooted tree $T$ ( compare  Figure 2.1). We associate to $T$ the polynomial
$Q(T,q)$ (or succinctly $Q(T)$) in the variable $q$ as follows.
\begin{enumerate}
\item[(i)] If $T$  is the one vertex tree, then $Q(T,q)=1$.
\item[(ii)] If $T$ has some edges (i.e. $|E(T)|>0$) then 
$$Q(T,q)=\sum_{v \in \mbox{ leaves }}q^{r(T,v)}Q(T-v,q), $$ 
where the sum is taken over all leaves, that is vertices of degree $1$ (not a root), of $T$
 and $r(T,v)$ is the number of edges of $T$ to the right of the unique path connecting $v$ with the root
 (an example is given in Figure 2.1).
\end{enumerate}
In Figure 2.2 we give an example of the expansion by our formula and if we complete the calculation we get the
polynomial $Q(T)$:
$$1+4q+9q^2+17q^3+28q^4+41q^5+56q^6+ 71q^7+83q^8+91q^9+94q^{10}+ $$
$$91q^{11}+ 83q^{12}+ 71q^{13}+ 56q^{14}+ 41q^{15}+ 28q^{16}+ 17q^{17}+9q^{18}+4q^{19}+q^{20}.\footnote{We show later that it is in 
fact equal to $[2]_q^2[4]_q{8\choose 3,5}_q;$ Corollary \ref{Corollary 2.3} (ii).}$$
\ \\
\centerline{\psfig{figure=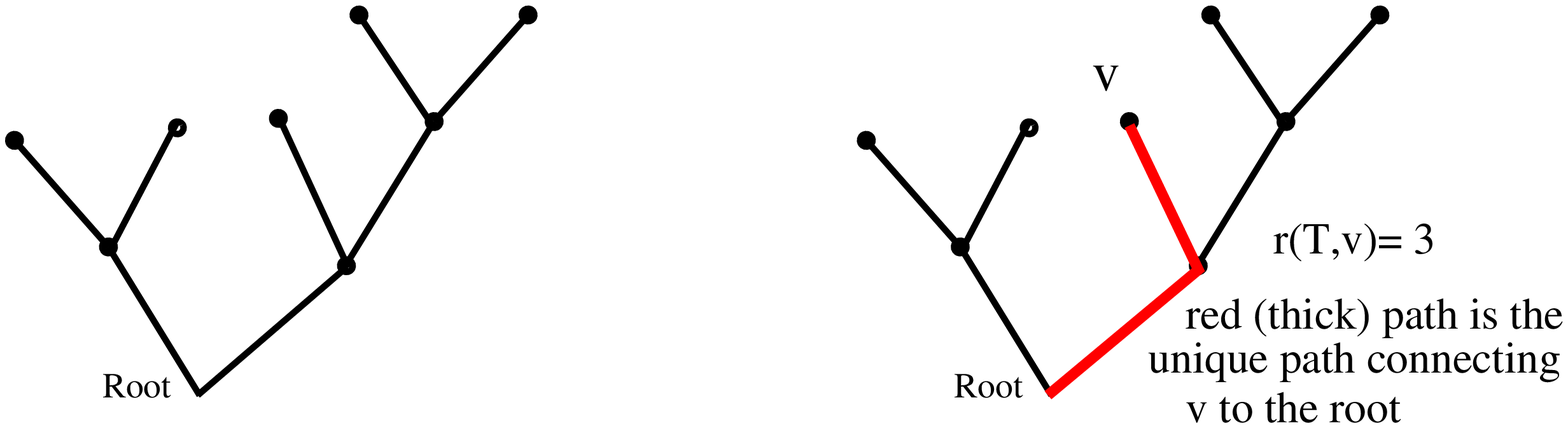,height=3.3cm}}
\ \\ \ \\
\centerline{ Figure 2.1; Plane rooted tree and an example of $r(T,v)$}
\ \\ \ \\

\centerline{\psfig{figure=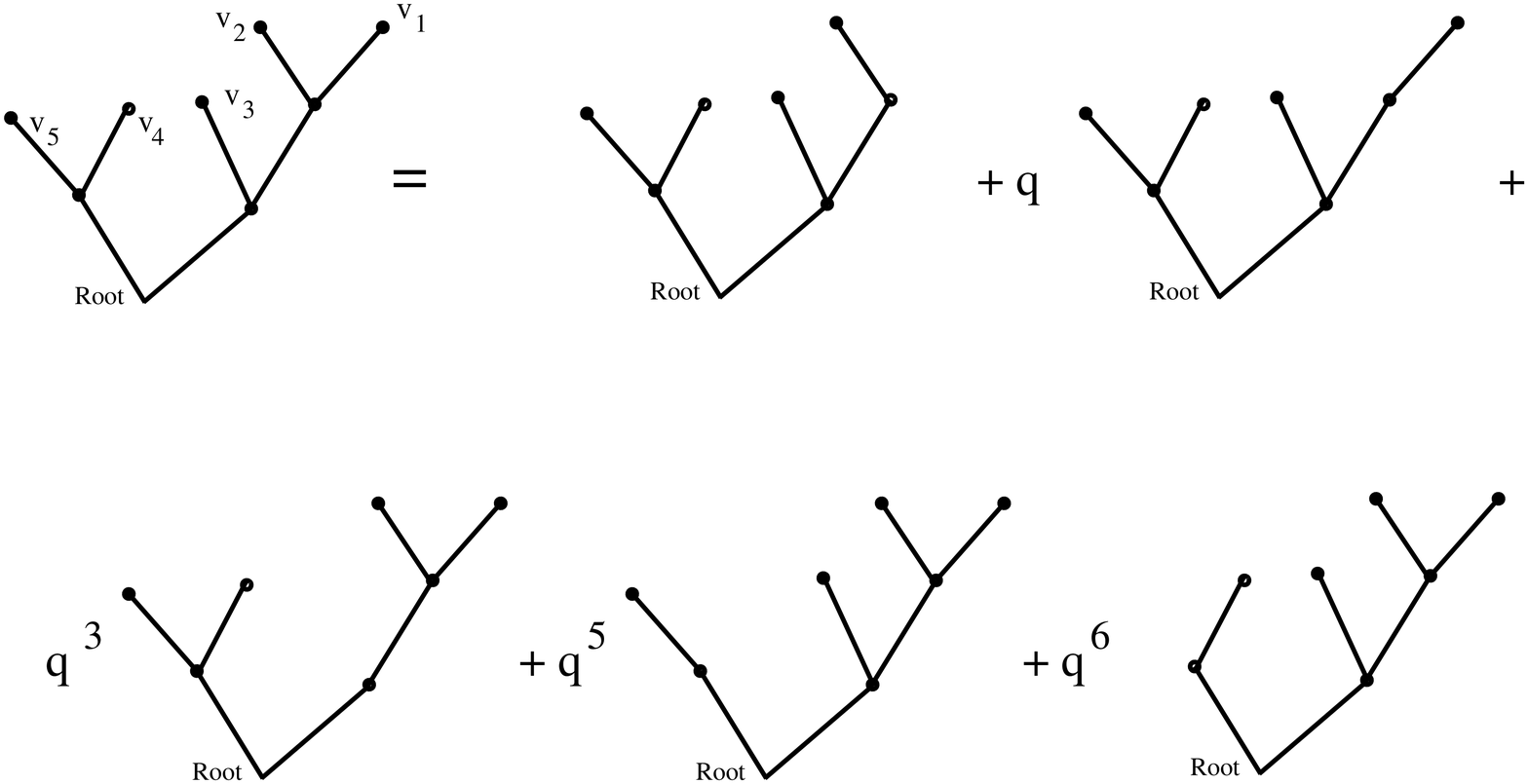,height=6.5cm}}
\ \\ \ \\
\centerline{ Figure 2.2; Tree expansion in the q-world}
\end{definition}

As an example we can check that for $T_{1,1,...,1}$ a star with $n$ rays we get $Q(T_{1,1,,...,1}) = [n]_q!$. 
In particular, $Q(\bigvee)= (1+q)=[2]_q$. One can look at  a proof of the formula by direct induction on $n$, but of course it is a 
very special case of the last formula of Section 1.

The first important result in our theory of $Q(T)$ polynomials is the product formula for trees glued along their roots (wedge or pointed product).
\begin{theorem}\label{Theorem 2.2}
Let $T_1 \vee T_2$ be a wedge product of trees $T_1$ and $T_2$ ({\psfig{figure=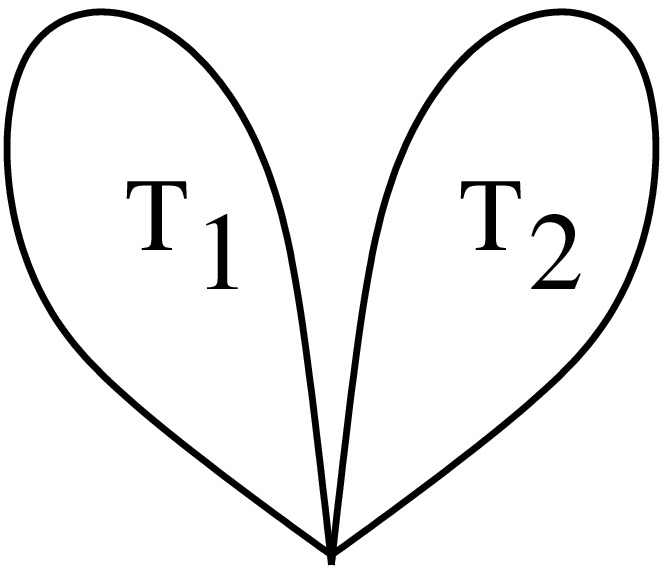,height=1.1cm}}). Then:
$$Q(T_1 \vee T_2)= \binom{|E(T_1)|+|E(T_2)|}{|E(T_1)|}_qQ(T_1)Q(T_2)$$
\end{theorem}

\begin{proof}
 We proceed by induction on the number of edges of $T$, $|E(T)|$, with the obvious initial case of no edges in one of the trees, that is
 $|E(T_1)|=0$ or $|E(T_2)|=0$. For simplicity we write $E_i$ for $|E(T_i)|$.\\
Let $T$ be a rooted plane tree with $E_1E_2\neq 0$, then we have:
$$Q(T)= \sum_{v\in L(T)}  q^{r(T,v)}Q(T-v)= $$
$$\sum_{v\in L(T_1)}  q^{r(T_1,v)+E_2 }Q(T_1-v)\vee T_2)+
\sum_{v\in L(T_2)}  q^{r(T_2,v)}Q(T_1\vee (T_2-v))  \stackrel{inductive\  assumption}{=}$$
$$\sum_{v\in L(T_1)}  q^{r(T_1,v)+E_2 }{E_1+E_2-1 \choose E_1-1,E_2}_q Q(T_1-v)Q(T_2) + $$
$$\sum_{v\in L(T_2)}  q^{r(T_2,v)}{E_1+E_2-1 \choose E_1,E_2-1}_qQ(T_1)Q(T_2-v) =$$
$$ Q(T_2)q^{E_2}{E_1+E_2-1\choose E_1-1,E_2}_q\sum_{v\in L(T_1)}  q^{r(T_1,v)}Q(T_1-v) +$$
$$ Q(T_1){E_1+E_2-1 \choose E_1,E_2-1}_q\sum_{v\in L(T_2)}  q^{r(T_2,v)}Q(T_2-v) \stackrel{definition}{=}$$
$$Q(T_1) Q(T_2)(q^{E_2}{E_1+E_2-1 \choose E_1-1,E_2}_q + 
{E_1+E_2-1\choose E_1,E_2-1}_q)=$$
$$ Q(T_1) Q(T_2){E_1+E_2\choose E_1,E_2}_q \mbox{ as needed}\footnote{One can modify the polynomial $Q(T)$ so that the formula of 
Theorem \ref{Theorem 2.2} can be interpreted as a homomorphism. For this we take $Q'(T)=\frac{Q(T)}{[|E(T)|]_q!}$. With this 
definition we have $Q'(T_1\vee T_2)= Q'(T_1)Q'(T_2)$. A disadvantage of such an approach is that $Q'(T)$ is not ncessarily 
a polynomial but only a rational function.}.$$
\end{proof}

Directly from Theorem \ref{Theorem 2.2} we conclude several properties of the $q$-polynomial, $Q(T)$, which by the nature of its definition, as 
pointed up before, we propose to be called the {\it plucking polynomial}. 
We should stress, in particular, part (iii) which establishes the independence of the polynomial of its plane embedding.

\begin{corollary}\label{Corollary 2.3}
\begin{enumerate}
\item[(i)] Let a plane tree be a wedge product of $k$ trees (\psfig{figure=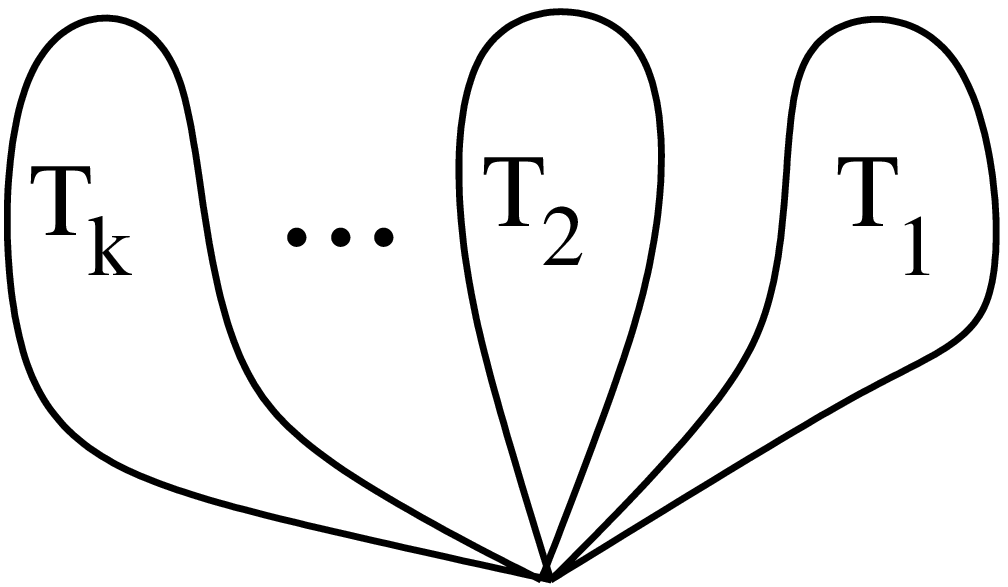,height=1.5cm}) that is 
$$T=T_{k} \vee ... \vee  T_2 \vee T_1,  \mbox{ then }$$
$$Q(T)= \binom{E_k+E_{k-1}+...+E_1}{E_k,E_{k-1},...,E_1}_qQ(T_k)Q(T_{k-1})\cdots Q(T_1),$$
where $E_i=|E(T_i)|$ is the number of edges in $T_i$.
\item[(ii)] (State product formula)
$$ Q(T) = \prod_{v\in V(T)}W(v), $$
where $W(v)$ is a weight of a vertex (we can call it a Boltzmann weight) defined by:
$$W(v)= \binom{E(T^v)}{E(T^v_{k_v}),...,E(T^v_{1})}_q,$$
where $T^v$ is a subtree of $T$ with root $v$ (part of $T$ above $v$, in other words $T^v$ grows from $v$)
and $T^v$ may be decomposed into wedge of trees as follows: $T^v= T^v_{k_v} \vee ... \vee  T^v_2 \vee T^v_{1}.$
\item[(iii)](Independence from a plane embedding) Plucking polynomial, $Q(T)$ does not depend on a plane embedding, it is therefore an 
invariant of rooted trees.
\item[(iv)] (Change of root). Let $e$ be be an edge of a tree $T$ with the endpoints  $v_1$ and $v_2$. Denote by $E_1$ 
 the number of edges of the tree $T_1$ with the root $v_1$ and $E_2$  the number of edges of the tree $T_2$ with the root $v_2$, where 
$$\mbox{ $T=$\psfig{figure=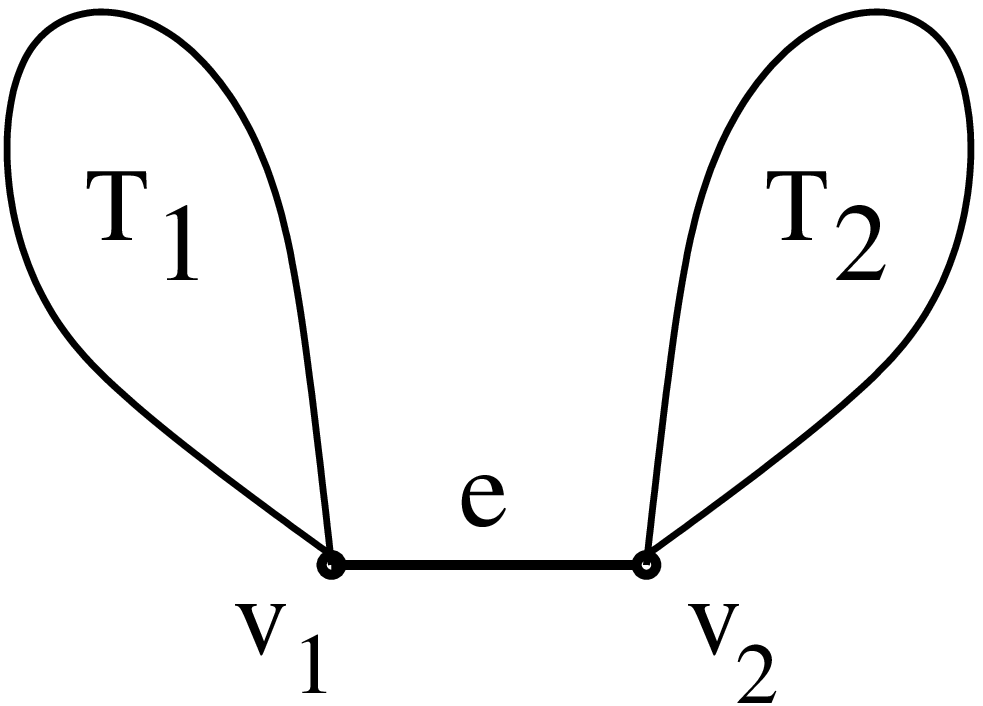,height=1.6cm}.  Then }
Q(T,v_1)=\frac{[E_1+1]_q}{[E_2+1]_q}Q(T,v_2).$$
\end{enumerate}
\end{corollary}
\begin{proof}
(i) The formula (i) follows by using $(k-1)$-times the product formula of Theorem 2.2 and the fact that the $q$-multinomial coefficient is a product of 
binomial coefficients:
$$\binom{a_k+a_{k-1}+...+a_2+a_1}{a_k,a_{k-1},...,a_2,a_1}_q \stackrel{def.}{=} 
\frac{[a_k+a_{k-1}+...+a_2+a_1]_q!}{[a_k]_q![a_{k-1}]_q!...[a_2]_q![a_1]_q!}=$$
$$\frac{[a_2+a_1]_q!}{[a_2]_q! [a_1]_q!}\cdot\frac{[a_3+a_2+a_1]_q!}{[a_3]_q! [a_2+a_1]_q!}\cdot ... \cdot
\frac{[a_k+a_{k-1}+...+a_2+a_1]_q!}{[a_k]_q! [a_{k-1}+...+a_2+a_1]_q!} =$$
$$= \binom{a_2+a_1}{a_2,a_1}_q\binom{a_3+a_2 + a_1}{a_3,a_2+a_1}_q\binom{a_4+a_3+a_2 + a_1}{a_4,a_3+a_2+a_1}_q...
\binom{a_k+a_{k-1}+...+a_2+a_1}{a_k,a_{k-1}+...+a_2+a_1}_q$$
(ii) The formula (ii) follows by using (i) several times.\\
(iii) Independence of embedding follows from the fact that the state product formula (ii) does not depend on the embedding.\\
(iv) We compare product formulas of Theorem  2.2 for $v_1$ and $v_2$ and we get:
$$Q(T,v_1)= {E_1+E_2+1 \choose E_1, E_2+1}_q Q(T_1)Q(T_2) \mbox{ and } $$ 
$$Q(T,v_2)= {E_1+E_2+1 \choose E_1+1, E_2}_q Q(T_1)Q(T_2),$$
from which the formula of (iv) follows directly.

\end{proof}

There are other nice properties of $Q(T)$ (e.g. the observation that it has polynomial time complexity) and I am sure readers will discover 
more. Here we give a few properties which have some importance in knot theory. 
\begin{corollary}\label{Corollary 2.4}
\begin{enumerate}
\item[(1)] $Q(T)$ is of the form $c_0+c_1q+...+c_Nq^N$ where:\\
(i) $c_0=1=c_N$, $c_i>0$ for every $i\leq N$, \\
(ii) $c_i=c_{N-i}$, that is  $Q(T)$ is a palindromic polynomial (often, less precisely we say symmetric polynomial),\\
(iii) the sequence $c_0,c_1,...,c_N$ is unimodal, that is for some $j$:
$$c_0\leq c_1 \leq ... \leq c_j \geq c_{j+1} \geq \ldots \geq c_N$$
(in our case $ j= \lfloor\frac{N}{2}\rfloor$ or $\lceil\frac{N}{2}\rceil$).\\
(iv) For a nontrivial tree $T$, that is a tree with at least one edge, we have:
$$c_1 = \sum_{v\in V(T)} (k_v-1),$$ where $k_v= deg_{T^v}(v)$ is the number of edges growing up from $v$, that is the degree of 
$v$ in the tree $T^v$ growing from $v$ (as in Corollary \ref{Corollary 2.3} (ii)). 
In particular, if $T$ is a binary tree, $c_1(T)$ is the number of vertices of $T$ which are not leaves.
\item[(2)] \begin{enumerate}
\item[(i)] $Q(T)$ is a product of  $q$-binomial coefficients (of type ${a+b\choose a}_q$).
\item[(ii)] $Q(T)$ is a product of cyclotomic polynomials\footnote{Recall that $n$th cyclotomic polynomial is a minimal polynomial 
which has as a root $e^{2\pi i/n}$. We can write this polynomial as: $\Psi_n(q)=\prod_{\omega^n=1,\omega^k\neq 1, k<n}(q-\omega)$.
For example $\Psi_4(q)=1+q^2$, $\Psi_6(q)=1-q+q^2$.}. 
\end{enumerate}
\item[(3)] The degree of the polynomial $N=deg Q(T)$ can be described by the formula:
$$N=deg Q(T)=\sum_{v\in V(T)}(\sum_{1\leq i < j \leq k_v}E_i^vE_j^v), $$
where, as in Corollary \ref{Corollary 2.3}(ii)  
$T^v$ is a subtree of $T$ with the root $v$ (part of $T$ above $v$, in other words $T^v$ is growing from $v$)
and $T^v$ can be presented as a bouquet of trees: $T^v= T^v_{k_v} \vee ... \vee  T^v_2 \vee T^v_{1}.$
\end{enumerate}
\end{corollary}
\begin{proof} 1(i) follows easily from the definition of the plucking polynomial. Namely we see that the constant term is obtained in a unique way 
by always plucking the most rightmost leaf from the tree (repeating this each time in the calculation). Thus $c_0=1$. 
Similarly, the highest power of $q$ is obtained uniquely by always plucking the leftmost leaf of the tree.\\
The condition that $c_i>0$ for any $i\leq N$ requires a more careful look at the recursive computation of $Q(T)$ and the proof is absolutely 
elementary; we leave it to the reader because in (iii) we provide much a stronger condition (but using a nontrivial fact proven by Sylvester).\\
(1)(ii) We start from observing the symmetry of $q$-binomial coefficients; namely we have: 
$$ {a+b \choose a,b}_{q^{-1}} = q^{-ab} {a+b \choose a,b}_q.$$
Then we use an easy observation that a product of symmetric polynomials is symmetric, and the fact we proved already and formulated 
in (2) that the polynomial $Q(T)$ is always a product of binomial coefficients.\\
(1)(iii) follows from the nontrivial fact, originally proved by Sylvester that $q$-binomial coefficients are
unimodal and from a simple observation that a product of symmetric (i.e. palindromic) positive unimodal polynomials is symmetric unimodal.
 (see \cite{Sta,Win}).\\
(1)(iv) The formula for $c_1$ follows from the observation that ${a+b \choose a, b}_q= 1+q+...$ for $a,b >0$ and more generally,
${a_1+a_2+...+a_k \choose a_1,a_2,...,a_k}_q = 1+ (k-1)q + ...$ for $a_1,a_2,...,a_k >0$.. Then the conclusion follows from the product formula 
of Corollary \ref{Corollary 2.3}.\\
(2) and (3): These conditions follow directly from the product formula of Corollary \ref{Corollary 2.3}(ii).
\end{proof}
With regard to Corollary \ref{Corollary 2.4}(iii) we can ask: for which trees, is $Q(T)$ strictly unimodal, that is 
$$c_0 < c_1 < ... < c_j > c_{j+1} > \ldots > c_N$$ for some $j$; compare  the computation for the tree of Figure 2.1 (see also \cite{Pak-Pan}).

\section{Comments and Connections}
As I was stressing from the beginning, and by now many readers may agree, the $q$-polynomial (plucking polynomial) is interesting on its own.
However, I would have never constructed or discovered it if I had not observed its shadow in my knot theory research. 
Concretely it was my work with M.Dabkowski and his student C.Li concerning skein modules of generalized (lattice) crossings, Figure 3.1, 
which gave some initial motivation, \cite{DLP}. In that paper we do not use $Q(T)$, as it was observed after the paper was completed. We will 
use it, however, in future research \cite{D-P}. Knot theory also motivates specific generalizations of the plucking polynomial by 
enhancing it with a delay function, $f: L(V) \to {\mathcal N}=\{n\in \Z\ | \ n \geq 1\}$. This function regulates which leaves are used in 
the recursive formula (Definition 2.1) and which are ``delayed". In fact we have much flexibility in the definition so we challenge a reader 
to play with possibilities.

The relation of the plucking polynomial to the Kauffman bracket skein modules is precise but difficult to describe succinctly. 
To have a concrete idea we can say that it concerns the study of the generalized (lattice) crossing (Figure 3.1), under the 
assumption that we resolve every crossing using the Kauffman bracket skein relation\\
 \parbox{5.3cm}{\psfig{figure=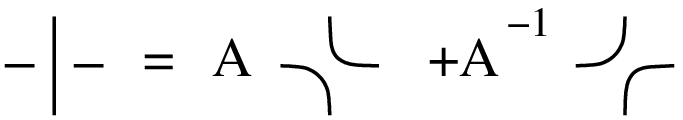,height=0.9cm}}, 
and replace every trivial component by the Laurent polynomial $-A^2 - A^{-2}$.

\centerline{\psfig{figure=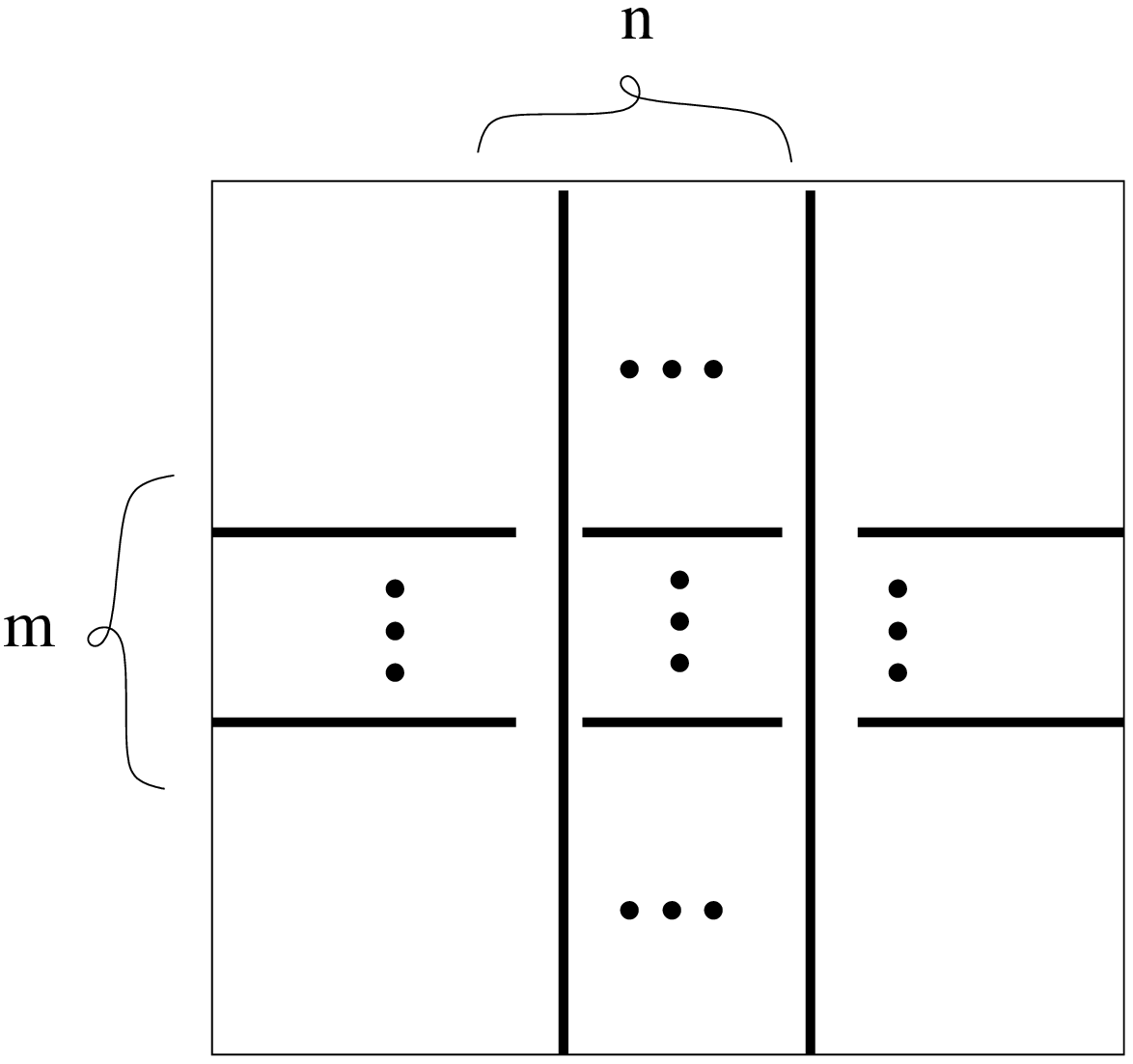,height=3.7cm}}\ \\
\centerline{ Figure 3.1; $T_{m\times n}$: $m\times n$ lattice crossing}
\ \\ \

One can generalize the plucking polynomial to any graph. If a graph $G$ has a base-point $b$ then the invariant $Q(G,b)$ is a 
collection (multiset) of plucking polynomials of spanning trees $G$ with the root $b$ \cite{Prz}. I challenge the reader to study 
connections of this invariant with known invariants of graphs.

Our polynomial also has relations to homological algebra:\\
let $\mathcal C$ be a chain complex, that is a sequence of abelian groups, $C_n$, and homomorphisms $\partial_n:C_n \to C_{n-1}$, so 
that $\partial_{n-1} \partial_n = 0$. On the basis of a chain complex we build  homology groups by defining 
$H_n(\mathcal C)= \ker \partial_n/im(\partial_{n+1})$.  Very often in topology and homological algebra the boundary operation $\partial_n$ is an 
alternating sum of homomorphisms, called face maps:
$\partial_n= \sum_{i=0}^n (-1)^id_i$. M.~Kapranov asked about what happens if  $(-1)^i$ is replaced by $q^n$, that is we define
 $$\partial^q_n= \sum_{i=0}^n q^id_i.$$ 
He noticed that if $q$ is a $k$th root of unity different from $1$ (i.e. $q^k=1, q\neq 1$) then the $k$th iteration of $\partial^q$ is the zero map
 ($\partial^q_{n-k+1}...\partial^q_{n-1} \partial^q_n = 0$), \cite{Kapr}.
The fact that Kapranov's idea is related to our $q$-polynomial is clear, however deep connections require  careful study: possibly you -- the reader -- 
could make a breakthrough.

\section{Acknowledgements}
J.~H.~Przytycki was partially supported by 
 the GWU REF grant, and Simons Collaboration Grant-316446.


\ \\ \ \\
Department of Mathematics,\\
The George Washington University,\\
Washington, DC 20052\\
e-mail: {\tt przytyck@gwu.edu},\\
University of Maryland College Park,\\
and University of Gda\'nsk

\newpage
{\Large \bf A free translation of the paper in Polish}
\ \\
\ \\
\ \\
\title{\bf P{\l}aszczyzna kwantowa i $q$-wielomian drzew z korzeniem}\\
\author{J\'ozef H. Przytycki}
\ \\ \ \\

\thispagestyle{empty}

\maketitle
\markboth{\hfil{\sc Wielomian drzew z korzeniem}\hfil}
\ \

{\bf Streszczenie}
Prezentujemy tutaj nowy niezmiennik drzew z korzeniem. S{\c a}dzimy, \.ze jest on nie tylko ciekawy sam w sobie, ale r\'ownie\.z poprzez
 zwi{\c a}zki z teori{\c a} w{\c e}z{\l}ow i algebr{\c a} homologiczn{\c a}. Jednak prawdziwym powodem prezentacji tego wielomianu czytelnikom
 Macierzatora jest to, \.ze mamy tu do czynienia z ca{\l}kowicie elementarn{\c a} niemniej zupe{\l}nie now{\c a} interesujac{\c a} matematyk{\c a}.
Po przeczytaniu tego eseju czytelnicy mog{\c a} bra\'c udzia{\l} w rozwoju niezmiennika, zastosowaniach i zwi{\c a}zkach z innymi dzia{\l}ami
matematyki czy te\.z mechaniki statystycznej lub kombinatorycznej biologii.

\tableofcontents

\setcounter{section}{0}
\section{Wst{\c e}p}\label{Section 1.1}

Nieprzemienna p{\l}aszczyzna zwana tak\.ze p{\l}aszczyzn{\c a} kwantow{\c a} rozwa\.zana ju\.z by{\l}a ponad sto lat temu przez
MacMahona\footnote{Percy Alexander MacMahon (1854 -- 1929).}. U\.zywa si{\c e}
jej w wielu dziedzinach matematyki i fizyki. Tutaj poka\.zemy jak naturalnie prowadzi ona do wielomianowego niezmiennika drzew i graf\'ow.

\subsection{Nieprzemienna p{\l}aszczyzna zwana tak\.ze p{\l}aszczyzn{\c a} kwantow{\c a}}\

Wszyscy znamy standardowy wz\'or na dwumian Newtona:
$$(x+y)^n= \sum_{i=0}^n {n \choose i} x^iy^{n-i}$$
Oczywi\'scie zak{\l}adamy tutaj przemienno\'s\'c zmiennych czyli $yx=xy$.
Mo\.zemy wyobrazi\'c sobie dow\'od bez s{\l}\'ow pisz{\c a}c
$$(x+y)^n =(x+y)(x+y)(x+y)...(x+y) \mbox{ $n$ razy} $$
i interpretuj{\c a}c wsp\'o{\l}czynnik przy jednomianie $x^iy^{n-i}$ jako liczb{\c e} wybor\'ow $i$ nawias\'ow z posr\'od $n$ nawias\'ow, z kt\'orych
wybieramy $x$ (z pozostai{\l}ych wybieramy $y)$. Liczb{\c a} t{\c a} jest ${n \choose i}$.

Je\'sli dowodzimy wz\'or dwumianowy przez indukcj{\c e} to wykorzystujemy to\.zsamo\'s\'c ${n \choose i}= {n-1 \choose i-1} + {n-1 \choose i}$,
t{\c e} sam{\c a} to\.zsamo\'s\'c z kt\'or{\c a} zwi{\c a}zany jest tr\'ojk{\c a}t Pascala.

Mo\.zna jednak zada\'c pytanie co si{\c e} stanie gdy przemienno\'s\'c
nieznacznie os{\l}abimy i za{\l}o\.zymy,
\.ze $yx=qxy$ ($q$ jest przemienne z $x$ i $y$). W zastosowaniach z fizyki $q$ jest cz{\c e}sto wybran{\c a}
liczb{\c a} zespolon{\c a} ale dla nas
najlepiej my\'sle\'c og\'olnie, \.ze $q$ jest zmienn{\c a} czyli pracujemy w
pier\'scieniu wielomian\'ow $Z[q]$.

Je\'sli nie znamy wyniku nale\.zy zacz{\c a}c od przyk{\l}ad\'ow:\\
$(x+y)^2= y^2+ xy +yx + x^2= y^2 + (1+q)xy +x^2,$ \\
$(x+y)^3= y^3+ xy^2+ yxy +y^2x + x^2y+ xyx + yx^2 +x^3 =$ \\
$y^3 + (1+q+q^2)xy^2 + (1+q+q^2)x^2y+ y^3,$\\

Widzimy, ze $1+q$ odgrywa rol{\c e} dw\'ojki w standardowym wzorze dwumianowym. Podobnie $1+q+q^2$ odgrywa rol{\c e} tr\'ojki.
Sugeruje to, \.ze rol{\c e} liczby $n$ odgrywa\'c b{\c e}dzie $1+q+q^2+...+q^{n-1}$, szczeg\'olnie, \.ze dla $q=1$ otrzymamy $n$.
Wprowad\'zmy wiec oznaczenie na ten kwantowy odpowiednik liczby $n$: $[n]_q= 1+q+q^2+...+q^{n-1}$\\
Liczmy dalej:
$(x+y)^4= y^4 + xy^3+ yxy^2 +y^2xy + y^3x + x^2y^2 + xyxy + xy^2x + yx^2y + yxyx + y^2x^2 +
          x^3y+ x^2yx + xyx^2 + yx^3 + x^4 =$
$ y^4 +  (1+q+q^2+q^3)xy^3 + (1+ q + 2q^2+ q^3 + q^4)x^2y^2 + (1+q+q^2+q^3)x^3y +x^4 =$
         $ y^4+(1+q+q^2+q^3)xy^3 + (1+q^2)(1+q+q^2)x^2y^2+ (1+q+q^2+q^3)x^3y +x^4 =$
$y^4 + [4]_qxy^3 + (1+q^2)[3]_q x^2y^2 + [4]_qx^3y + y^4$.\\
Tutaj naturaln{\c a}, zaczyna by\'c sugestia, \.ze  nie tylko $1+q+...q^{n-1}$ powinno byc $q$-odpowiednikiem liczby $n$,
ale  tym samym duchu powinnismy zdefiniowa\'c silni{\c e} $q$-liczby $[n]_q$ jako: $[n]_q!=[n]_q[n-1]_q\cdots [2]_q[1]_q$, a
$q$-odpowiednik symbolu Newtona, zwany tak\.ze $q$-wielomianem Gaussa jako $\binom{n}{i}_q = \frac{[n]_q!}{[i]_q! [n-i]_q!}$
(cz{\c e}sto by podkre\'sli\'c symetri{\c e} $q$-wielomianu Gaussa b{\c e}dziemy te\.z u\.zywa\'c
zapisu $\binom{n}{i,n-i}_q$).
Mozemy wtedy zapisa\'c wsp\'o{\l}czynnik przy $x^2y^2$ jako:\\
$(1+q^2)[3]_q= \frac{[4]_q[3]_q}{[2]_q} = \frac{[4]_q!}{[2]_q![2]_q!}= {4 \choose 2}_q$.

W naszej notacji poprzednie rachunki mog{\c a} by\'c zapisane zwi{\c e}\'zle jak poni\.zej:
$$(x+y)^2= y^2 +[2]_qxy +x^2,$$
$$(x+y)^3= y^3 + [3]_qxy^2 + [3]_qx^2y+x^3,$$
$$(x+y)^4= y^4+ [4]_qxy^3 + \frac{[4]_q[3]_q}{[2]_q}x^2y^2 + [4]_qx^3y+ x^4 =$$
$$ y^4+ [4]_qxy^3 + \binom{4}{2}_q x^2y^2 + [4]_qx^3y+ x^4.$$

Mo\.zna teraz zgadn{\c a}\'c og\'oln{\c a} formu{\l}{\c e}:
\begin{stwierdzenie}\label{Stwierdzenie N.5.3}
$$(x+y)^n= \sum_{i=0}^n {n \choose i}_qx^iy^{n-i}.$$
\end{stwierdzenie}
{\bf Dow\'od.} Dow\'od bez s{\l}\'ow zostawiam czytelnikom, mimo \.ze jest podobny do przemiennego przypadku wymaga wi{\c e}kszego skupienia i ja
go widz{\c e} gdy jestem wyspany i po kawie...

Nietrudno zasugerowa\'c jednak dow\'od przez indukcj{\c e} po $n$. Sprawdzili\'smy ju\.z formu{\l}{\c e} dla $n\leq 4$ (by formu{\l}a dzia{\l}a{\l}a
potrzebujemy, jak w klasycznym przypadku, konwencji, \.ze $[0]_q!=1$ i w konsekwencji ${n \choose 0}_q=1 ={n \choose n}_q$.\\
Teraz przeprowadzimy krok indukcyjny (od $n-1$ do $n$):
$$(x+y)^n= (x+y)(x+y)^{n-1}=(x+y)\sum_{i=0}^{n-1} {n-1 \choose i,n-i-1}_qx^iy^{n-i-1}=$$
$$\sum_{i=0}^{n-1} {n-1 \choose i,n-i-1}_qx^{i+1}y^{n-i-1} +\sum_{i=0}^{n-1}q^i{n-1 \choose i,n-i-1}_qx^iy^{n-i} =$$
$$\sum_{i=0}^n({n-1 \choose i-1,n-i}_q +q^i{n-1 \choose i,n-i-1}_q)x^iy^{n-i}.$$
U\.zywamy konwencji, \.ze ${n-1 \choose -1,n}_q=0$. \\
Pozostaje nam sprawdzi\'c, \.ze $${n-1 \choose i-1,n-i}_q +q^i{n-1 \choose i,n-i-1}_q={n \choose i}_q.$$
Zach{\c e}camy czytelnika do sprawdzenia tego samemu bez patrzenia na nastepuj{\c a}ce obliczenie:
\begin{enumerate}
\item[(i)] $[a+b]_q= 1+q+...+q^{a+b-1}= (1+q+...+q^{a-1}) + q^a(1+q+...+q^{b-1})= [a]_q+q^a[b]_q= [b]_q+q^b[a]_q$.
\item[(ii)] $$\binom{a+b}{a,b}_q=\frac{[a+b]_q!}{[a]_q! [b]_q!}=[a+b]_q\frac{[a+b-1]_q!}{[a]_q! [b]_q!}=$$
$$([a]_q+q^a[b]_q)\frac{[a+b-1]_q!}{[a]_q! [b]_q!}= \frac{[a+b-1]_q!}{[a-1]_q! [b]_q!} + q^a \frac{[a+b-1]_q!}{[a]_q! [b-1]_q!}= $$
$$\binom{a+b-1}{a-1,b}_q + a^q\binom{a+b-1}{a,b-1}_q= \binom{a+b-1}{a,b-1}_q + b^q\binom{a+b-1}{a-1,b}_q.$$
\end{enumerate}
\ \\ \ \\

Mo\.zemy interpretowa\'c klasyczn{\c a} formu{\l}e na wsp\'o{\l}czynniki $(x+y)^n$ przez rozwa\.zenie drzewa $T_{b,a}$ o dw\'och d{\l}ugich kraw{\c e}dziach
(lub raczej ga{\l}{\c e}ziach)
o d{\l}ugo\'sci $b$ i $a$ odpowiednio, wychodz{\c a}cych z korzenia, tak jak to jest pokazane na rysunku 1.1.
Mo\.zna zada\'c pytanie na ile sposob\'ow drzewo to da si{\c e} oskuba\'c jak w dziecinnej zabawie ``lubi -- nie lubi",
li\'s\'c po li\'sciu. Za ka\.zdym razem musimy tylko zdecydowa\'c
czy skubiemy lew{\c a} czy praw{\c a} ga{\l}{\c a}\'z; lewa ma by\'c skubana $b$ razy a prawa $a$ razy. Jasne, \.ze musimy wybra\'c
$a$ spo\'sr\'od $a+b$ ruch\'ow tak\.ze
odpowied\'z to ${a+b \choose a}$. Nasza praca opowiada o dramatycznym uog\'olnieniu tego przyk{\l}adu.

\ \\
\centerline{\psfig{figure=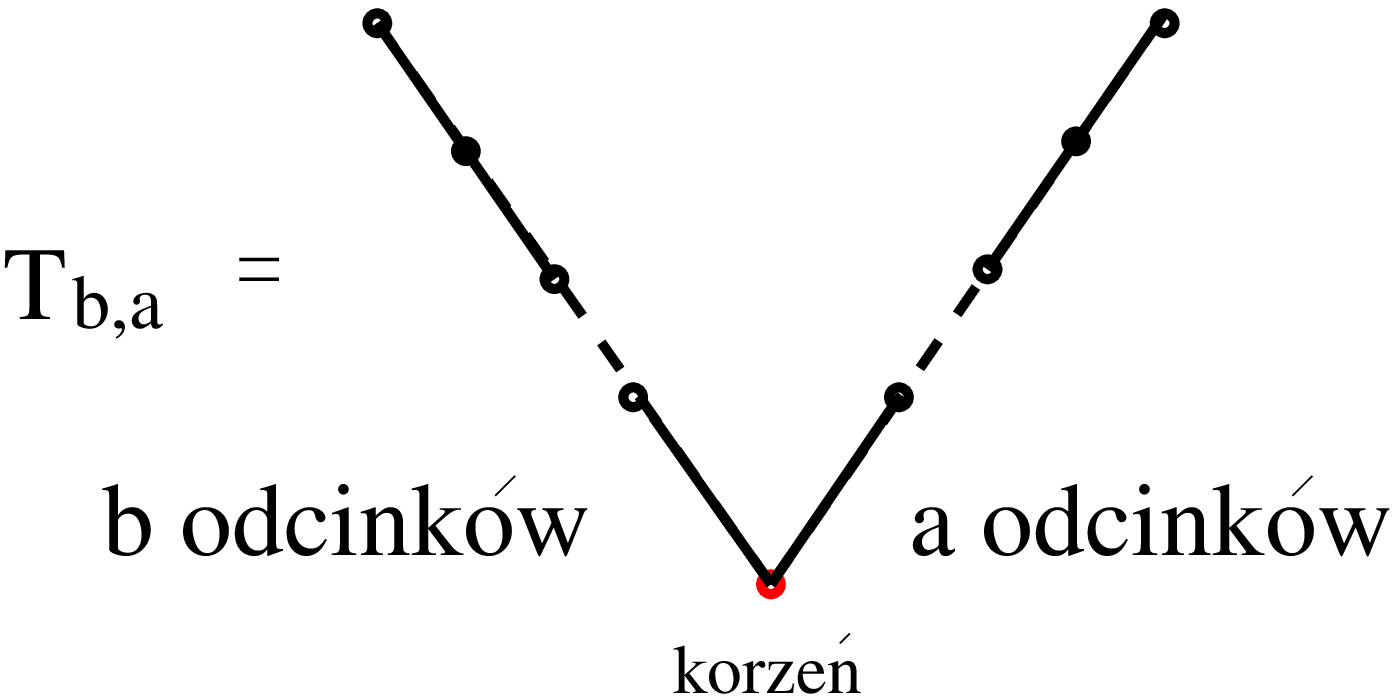,height=3.5cm}}\ \\
\centerline{Rysunek 1.1; drzewo $T_{b,a}$ z ga{\l}{\c e}ziami o d{\l}ugo\'sci  $b$ i $a$}\ \\
\ \\

Wr\'o\'cmy do naszego drzewa o dw\'och ga{\l}{\c e}ziach, i spr\'obujmy gra\'c gr{\c e} skubania w $q$-stylu:
Znaczy\'c to b{\c e}dzie, \.ze ka\.zde skubanie b{\c e}dzie mia{\l}o wag{\c e} $q$ podniesion{\c a} do jakiej\'s pot{\c e}gi.
Dok{\l}adniej, li\'s\'c z prawej strony b{\c e}dzie mia{\l} wag{\c e} $1$ a li\'s\'c z lewej strony wag{\c e} $q^a$ (jako przyczyn{\c e} takiego wyboru podamy
to, \.ze ma on $a$ odcink\'ow (kraw{\c e}dzi) po swej prawej stronie).
Tak wiec jako wynik otzymamy nie liczb{\c e} ``skuba\'n" ale wielomian skuba\'n, kt\'ory oznaczymy przez $Q(T_{b,a})$.
Nasza definicja daje nast{\c e}puj{\c a}c{\c a} rekurencyjn{\c a} relacj{\c e}: $Q(T_{b,a})= Q(T_{b,a-1}) + q^aQ(T_{b-1,a})$.
Zauwai\.zamy tak\.ze, \.ze $Q(T_{0,a})=Q(T_{b,0})=1$. Mo\.zemy teraz rozpozna\'c, \.ze
wielomian skubania dany jest formu{\l}{\c a}:  $Q(T_{b,a})= \binom{a+b}{a,b}_q$.
To jest punkt wyj\'sciowy do og\'olnej definicji $q$-wielomianu drzewa z korzeniem o kt\'orym mowa w tytule tego eseju.

Mo\.zemy powt\'orzyc nasze rozwa\.zania u\.zywaj{\c a}c wielu zmiennych, $x_1,x_2,...,x_k$. Je\'sli zmienne s{\c a} przemienne to
otrzymujemy znan{\c a} nam pewnie wszystkim z rachunku r\'o\.zniczkowego wielu zmiennych wielomianow{\c a} formu{\l}{\c e} Newtona:
$$(x_1+x_2+...x_k)^n= \sum_{a_1,...,a_k; \sum a_i=n}^n {n \choose a_1,...,a_k}x_1^{a_1} x_2^{a_2}\cdots x_k^{a_k},$$
gdzie ${n \choose a_1,...,a_k} = \frac{n!}{a_1!...a_k!}$. Tak jak przedtem mi{\l}{\c a} interpretacj{\c a} tego wyra\.zenia
 ${n \choose a_1,...,a_k}$ jest liczba skuba\'n drzewa $T_{a_k,...,a_2,a_1}$ o $k$  d{\l}ugich ga{\l}{\c e}ziach d{\l}ugo\'sci
 $a_k,...,a_2,a_1$ tak jak to jest pokazane na rysunku 1.2 poni\.zej.

\ \\
\centerline{\psfig{figure=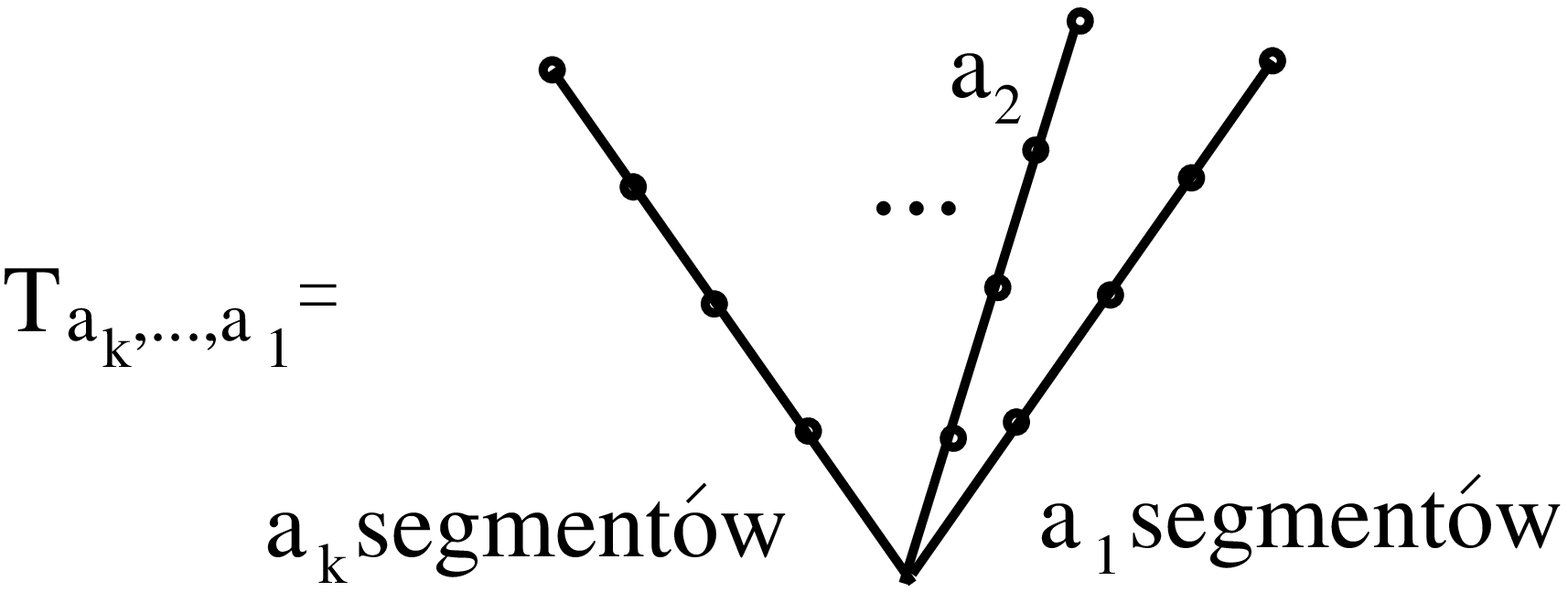,height=4.5cm}}\ \\
\centerline{Rysunek 1.2; drzewo $T_{a_k,...,a_1}$ z ga{\l}{\c e}ziami d{\l}ugo\'sci $a_k,...,a_1$}\ \\

Mo\.zemy teraz rozpatrzy\'c nieprzemienn{\c a} przestrze\'n z $x_jx_i=qx_ix_j$ dla $i<j$.
Otrzymamy teraz formu{\l}{\c e}:
$$(x_1+x_2+...x_k)^n= \sum_{a_1,...,a_k; \sum a_i=n}^n {n \choose a_1,...,a_k}_qx_1^{a_1} x_2^{a_2}\cdots x_k^{a_k},$$
gdzie ${n \choose a_1,...,a_k}_q = \frac{[n]_q!}{[a_1]_q!...[a_k]_q!}$.
\ \\
Aby wykaza\'c t{\c e} formu{\l}{\c e} mo\.zemy, jak przedtem, u\.zy\'c indukcji po $n$ i nastepuj{\c a}ce to\.zsamo\'sci s{\c a} kluczowe w dowodzie:
\begin{enumerate}
\item[(i)] $[a_1+a_2+...+a_k]_q= [a_1]_q + q^{a_1}[a_2]_q + q^{a_1+a_2}[a_2]_q+...+q^{a_1+a_2+...+a_{k-1}}[a_k]_q.$
\item[(ii)] $${a_1+a_2+...+a_k \choose a_1,a_2,...,a_k}_q = 
{a_1+a_2+...+a_k-1 \choose a_1-1,a_2,...,a_k}_q + $$
$$ q^{a_1}{a_1+a_2+...+a_k-1 \choose a_1,a_2-1,...,a_k}_q+...
+ q^{a_1+a_2+...+a_{k-1}} {a_1+a_2+...+a_k-1 \choose a_1,a_2,...,a_k-1}_q.$$
\end{enumerate}

Ponownie mo\.zemy interpretowa\'c $q$-wielomianowe wsp\'o{\l}czynniki poprzez $q$-skubanie drzewa   $T_{a_k,...,a_2,a_1}$,  zak{\l}adaj{\c a}c
przy $q$-skubaniu nast{\c e}puj{\c a}c{\c a} formu{\l}{\c e}:
$$Q(T_{a_k,...,a_2,a_1}) = Q(T_{a_k,...,a_2,a_1-1}) + q^{a_1}Q(T_{a_k,...,a_2-1,a_1}) +$$
$$ q^{a_1+a_2}Q(T_{a_k,...,a_3-1,a_2,a_1}) + ... + q^{a_1+a_2+...+a_{k-1}}Q(T_{a_k-1,...,a_2,a_1}).$$
Poniewa\.z rekurencyjny wz\'or jest taki sam jak dla $q$ wielomianowych symboli Newtona wi{\c e}c
 $$Q(T_{a_k,...,a_2,a_1}) = {a_1+a_2+...+a_k \choose a_1,a_2,...,a_k}_q.$$

\section{Rekurencyjna definicja  $q$-wielomianu p{\l}askiego drzewa z korzeniem}\

Poka\.zemy tutaj jak nasze rozwa\.zania o kwantowej p{\l}aszczyznie i nieprzemiennej przestrzeni mo\.zna uog\'olni\'c
do dowolnego drzewa z korzeniem (tzn. punktem wyr\'o\.znionym czy bazowym). Zaczniemy
od za{\l}o\.zenia, \.ze drzewo jest p{\l}askie (to znaczy zanurzone na p{\l}aszczy\'znie), a p\'ozniej wyka\.zemy, \.ze otrzymany $q$-wielomian skubania
nie zale\.zy od w{\l}o\.zenia. Mo\.zemy w{\c e}c powiedzie\'c, \.ze $q$-wielomian jest niezmiennikiem drzewa z korzeniem.

\ \\ \ \\ \
Tak wi{\c e}c zak{\l}adamy w tej  konstrukcji, \.ze drzewo z korzeniem
 jest zanurzone na p{\l}aszczyznie (p{\l}askie drzewo). W naszej konwencji drzewo ro\'snie do g\'ory.
\begin{definition}\label{Definicja N.9.2}
Niech $T$ b{\c e}dzie p{\l}askim drzewem z korzeniem $v_0$, wtedy wieloman $Q_{T,v_0}(q)$ lub kr\'otko $Q(T)\in \Z[q]$ jest zdefiniowany
przez warunek pocz{\c a}tkowy $Q(\bullet)=1$ i relacj{\c e} rekurencyjn{\c a}
$$Q(T)= \sum_{v\in L(T)}q^{r(T,v)}Q(T-v), \mbox{ gdzie $L(T)$ jest zbiorem li\'sci drzewa $T$, }$$
to znaczy wierzcho{\l}k\'ow stopnia $1$ r\'o\.znych od korzenia,
a  $r(T,v)$ jest liczb{\c a} kraw{\c e}dzi $T$ na prawo od jedynej drogi {\l}{\c a}cz{\c a}cej $v$ z korzeniem $v_0$; por\'ownaj rysunek 2.1.
\end{definition}

\centerline{\psfig{figure=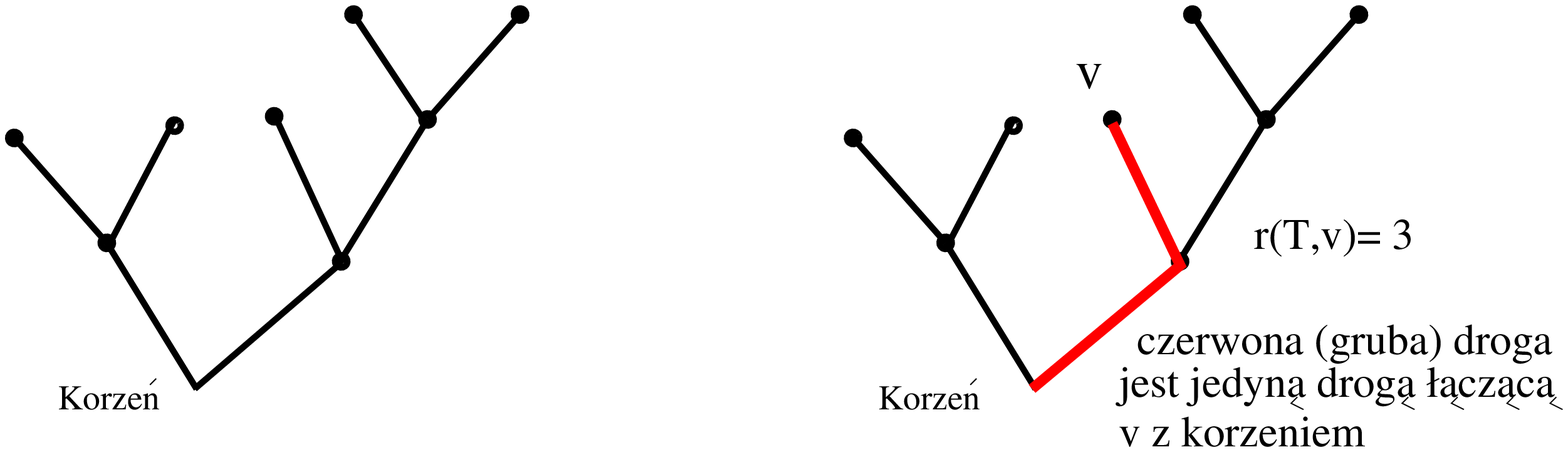,height=2.7cm}}
 \ \\
\centerline{ Rysunek 2.1; p{\l}askie drzewo z korzeniem i przyk{\l}ad liczenia wyk{\l}adnika $r(T,v)$}
\ \\ \ \\

Dla przyk{\l}adu policzyli\'smy ju\.z, \.ze $Q(\bigvee)= (1+q)=[2]_q$ lub og\'olniej $Q(T_n)= [n]_q!$,
gdzie $T_n$ jest gwiazd{\c a} z  $n$ promieniami,  $[n]_q=1+q+...+q^{n-1}$
a $q$-silnia zdefiniowana jest analogicznie jak zwyk{\l}a silnia:
  $[n]_q!= [n]_q[n-1]_q\cdots [2]_q[1]_q$ (czyli notacja jakiej u\.zywali\'smy w przypadku kwantowej powierzchni).

\begin{twierdzenie}\label{Twierdzenie 2.2}
Niech $T_1 \vee T_2$ b{\c e}dzie bukietem  (produktem sklejajacym korzenie))
({\psfig{figure=Wedge-product.eps,height=1.1cm}}). Wtedy:
$$Q(T_1 \vee T_2)= \binom{|E(T_1)|+|E(T_2)|}{|E(T_1)|}_qQ(T_1)(Q(T_2)$$
\end{twierdzenie}
{\bf Dow\'od.}
Dow\'od przeprowadzamy przez indukcj{\c e} po ilo\'sci kraw{\c e}dzi , $|E(T)|$, drzewa $T=T_1\vee T_2$ z naturalnym warunkiem pocz{\c a}tkowym
gdy jedno z drzew nie ma kraw{\c e}dzi ($|E(T_1)||E(T_2)|=0$) a formu{\l}a zachodzi w spos\'ob oczywisty.
Dla prostoty piszemy $E_i$ zamiast $|E(T_i)|$.\\
Niech $T$ b{\c e}dzie p{\l}askim drzewem z korzeniem i  $E_1E_2> 0$. Wtedy otrzymujemy:
$$Q(T)= \sum_{v\in L(T)}  q^{r(T,v)}Q(T-v)= $$
$$\sum_{v\in L(T_1)}  q^{r(T_1,v)+E_2 }Q((T_1-v)\vee T_2)+
\sum_{v\in L(T_2)}  q^{r(T_2,v)}Q(T_1\vee (T_2-v))  \stackrel{\mbox{za{\l}o\.zenie\ indukcyjne}}{=}$$
$$\sum_{v\in L(T_1)}  q^{r(T_1,v)+E_2}{E_1+E_2-1 \choose E_1-1,E_2}_q Q(T_1-v)Q(T_2) + $$
$$\sum_{v\in L(T_2)}  q^{r(T_2,v)}{E_1+E_2-1 \choose E_1,E_2-1}_qQ(T_1)Q(T_2-v) =$$
$$ Q(T_2)q^{E_2}{E_1+E_2-1\choose E_1-1,E_2}_q\sum_{v\in L(T_1)}  q^{r(T_1,v)}Q(T_1-v) +$$
$$ Q(T_1){E_1+E_2-1 \choose E_1,E_2-1}_q\sum_{v\in L(T_2)}  q^{r(T_2,v)}Q(T_2-v) =$$
$$Q(T_1) Q(T_2)(q^{E_2}{E_1+E_2-1 \choose E_1-1,E_2}_q +
{E_1+E_2-1\choose E_1,E_2-1}_q)=$$
$$ Q(T_1) Q(T_2){E_1+E_2\choose E_1,E_2}_q \mbox{ jak potrzeba}.\footnote{Mo\.zna zmodyfikowa\'c wielomian 
$Q(T)$ w ten spos\'ob, \.ze formu{\l}a z Twierdzenia \ref{Twierdzenie 2.2} stanie si{\c e} homomorfizmem. W tym celu definiujemy
$Q'(T)=\frac{Q(T)}{[|E(T)|]_q!}$.  Mamy teraz  $Q'(T_1\vee T_2)= Q'(T_1)Q'(T_2)$. Wad{\c a} takiej definicji  jest to, \.ze $Q'(T)$ nie 
zawsze jest wielomianem a cz{\c e}sto jest tylko funkcj{\c a} wymiern{\c a}.}$$
\rule{1.6ex}{1.6ex}
\begin{wniosek}\label{Wniosek 2.3}
\begin{enumerate}
\item[(i)] Je\'sli p{\l}askie drzewo z korzeniem jest bukietem $k$ drzew
 (\psfig{figure=Tree-seq-vee.eps,height=1.5cm}) to znaczy
$$T=T_{k} \vee ... \vee  T_2 \vee T_1,  \mbox{ wtedy }$$
$$Q(T)= \binom{E_k+E_{k-1}+...+E_1}{E_k,E_{k-1},...,E_1}_qQ(T_k)Q(T_{k-1})\cdots Q(T_1),$$
gdzie $E_i=|E(T_i)|$ jest liczb{\c a} kraw{\c e}dzi $T_i$.
\item[(ii)] (Formu{\l}a w postaci produktu po stanach)
$$ Q(T) = \prod_{v\in V(T)}W(v), $$
gdzie $W(v)$ jest wag{\c a} wierzcho{\l}ka (mo\.zemy nazywac to wag{\c a} Boltzmanna) zdefiniowan{\c a} przez:
$$W(v)= \binom{E(T^v)}{E(T^v_{k_v}),...,E(T^v_{1})}_q,$$
gdzie $T^v$ jest poddrzewem $T$ z korzeniem $v$ (czesc $T$ powy\.zej $v$, innymi s{\l}owy $T^v$ wyrasta z $v$)
i $T^v$ mo\.ze byc roz{\l}o\.zone na bukiet drzew: $T^v= T^v_{k_v} \vee ... \vee  T^v_2 \vee T^v_{1}.$
\item[(iii)](niezale\.zno\'s\'c od w{\l}o\.zenia) Wielomian skubania, $Q(T)$ nie zale\.zy od w{\l}o\.zenia jest wi{\c e}c
niezmiennikiem drzewa z korzeniem.
\item[(iv)] (zmiana korzenia). Niech $e$ bedzie kraw{\c e}dzi{\c a} drzewa $T$ z koncami $v_1$ i
$v_2$ a $E_1$ jest liczba kraw{\c e}dzi od strony $v_1$ drzewa, a $E_2$ jest liczba kraw{\c e}dzi $T$
po stronie $v_2$ kraw{\c e}dzi $e$, gdzie
$$\mbox{ $T=$\psfig{figure=Tree-change-base.eps,height=1.6cm}. Wtedy }
Q(T,v_1)=\frac{[E_1+1]_q}{[E_2+1]_q}Q(T,v_2).$$
\end{enumerate}
\end{wniosek}

{\bf Dow\'od.}
(i) Wz\'or z (i)  wynika przez u\.zycie kilka razy formu{\l}y:
$$Q(T_2 \vee T_1)= {E(T_2)+E(T_1) \choose E(T_2),E(T_1)}_qQ(T_2)Q(T_1),$$
jako, \.ze mamy:
$$\binom{a_k+a_{k-1}+...+a_2+a_1}{a_k,a_{k-1},...,a_2,a_1}_q=
\binom{a_{k-1}+...+a_2+a_1}{a_{k-1},...,a_2,a_1}_q \binom{a_k+a_{k-1}+...+a_2+a_1}{a_k,a_{k-1}+...+a_2+a_1}_q=...$$
$$= \binom{a_2+a_1}{a_2,a_1}_q\binom{a_3+a_2 + a_1}{a_3,a_2+a_1}_q\binom{a_4+a_3+a_2 + a_1}{a_4,a_3+a_2+a_1}_q...
\binom{a_k+a_{k-1}+...+a_2+a_1}{a_k,a_{k-1}+...+a_2+a_1}_q$$
(ii) Formula (ii) wynika przez u\.zycie (i) wiele razy.\\
(iii) Niezale\.zno\'s\'c od w{\l}o\.zenia wynika z faktu, \.ze formu{\l}a produktowa z (ii) nie zale\.zy od w{\l}o\.zenia.\\
(iv) Por\'ownujemy formu{\l}e z Twierdzenia 2.2 dla $v_1$ i $v_2$ otrzymuj{\c a}c:
$$Q(T,v_1)= {E_1+E_2+1 \choose E_1, E_2+1}_q Q(T_1)Q(T_2) \mbox{ oraz } $$
$$Q(T,v_2)= {E_1+E_2+1 \choose E_1+1, E_2}_q Q(T_1)Q(T_2);$$
teraz formu{\l}a zmiany korzenia wynika natychmiast.

\rule{1.6ex}{1.6ex}\\

Podamy teraz kilka wniosk\'ow opisuj{\c a}cych proste bazowe w{\l}asno\'sci $Q(T)$ i maj{\c a}cych
pewne znaczenie w teorii w{\c e}z{\l}\'ow.
\begin{wniosek}\label{Wniosek 2.4}
\begin{enumerate}
\item[(1)] $Q(T)$ jest postaci $c_0+c_1q+...+c_Nq^N$ gdzie:\\
(i) $c_0=1=c_N$, $c_i>0$ dla ka\.zdego $i\leq N$ oraz\\
(ii) Niech $k_v$ oznacza liczb{\c e} kraw{\c e}dzi wyrastaj{\c a}cych z wierzcho{\l}ka $v$ do g\'ory, inaczej m\'owi{\c a}c
jest to stopie\'n wierzcho{\l}ka $v$ w drzewie $T^v$ wyrastaj{\c a}cym z $v$. W tej notacji, dla drzewa z korzeniem $T$ maj{\c a}cym
cho\'c jedn{\c a} kraw{\c e}d\'z, wsp\'o{\l}czynnik $c_1$ dany jest nast{\c e}puj{\c a}c{\c a} formu{\l}{\c a}:
$$c_1 = \sum_{v\in V(T)} (k_v-1).$$
W szczeg\'olno\'sci w przypadku nietrywialnego drzewa binarnego, dla ka\.zdego wierzcho{\l}ka r\'o\.znego od li\'scia
mamy $k_v= 2$ a wi{\c e}c $c_1 = |V(T)|-|L(T)|$.\\
(iii) $c_i=c_{N-i}$ (tzn. $Q(T)$ jest wielomianem symetrycznym (palindromicznym)).
\item[(2)] \begin{enumerate}
\item[(i)] $Q(T)$ jest iloczynem $q$-symboli dwumianowych (typu ${a+b\choose a}_q$).
\item[(ii)] $Q(T)$ jest iloczynem wielomianow cyklotomicznych\footnote{Przypomnijmy, \.ze $n$-ty wielomian cyklotomiczny
jest minimalnym wielomianem, kt\'orego pierwiastkiem jest $e^{2\pi i/n}$.
Mo\.zemy go zapisa\'c: $\Psi_n(q)=\prod_{\omega^n=1,\omega^k\neq 1, k<n}(q-\omega)$.
Dla przyk{\l}adu $\Psi_4(q)=1+q^2$, $\Psi_6(q)=1-q+q^2$.}.
\end{enumerate}
\item[(3)] Stopie\'n wielomianu $N=deg Q(T)$ mo\.zna zapisa\'c formu{\l}{\c a}:
$$N=deg Q(T)=\sum_{v\in V(T)}(\sum_{1\leq i < j \leq k_v}E_i^vE_j^v), $$
gdzie, jak we wniosku \ref{Wniosek 2.3}(ii),
$T^v$ jest poddrzewem $T$ z korzeniem $v$ (cze\'s\'c $T$ powy\.zej $v$, innymi s{\l}owy $T^v$ wyrasta z $v$)
i $T^v$ mo\.ze by\'c roz{\l}o\.zone na bukiet drzew: $T^v= T^v_{k_v} \vee ... \vee  T^v_2 \vee T^v_{1}.$
\end{enumerate}
\end{wniosek}

{\bf Dow\'od.}
1(i) naj{\l}atwiej wykaza\'c z definicji; nietrudno zobaczy\'c, \.ze wyraz sta{\l}y otrzymujemy w jednoznaczny spos\'ob bior{\c a}c
li\'s\'c najbardziej z prawej strony drzewa i powtarzaj{\c a}c to w ka\.zdym kroku obliczen; stad $c_0=1$.
Podobnie, najwy\.zsz{\c a} pot{\c e}g{\c e} $q$  otrzymujemy bior{\c a}c
li\'s\'c polo\.zony najbardziej z lewej strony drzewa w definicji rekurencyjne $Q(T)$, i ze wszystkie inne wybory dadz{\c a}  mniejsz{\c a}
pot{\c e}g{\c e}.\\
Warunek $c_i>0$ dla ka\.zdego $i\leq n$  wymaga uwa\.zniejszego przyjrzenia si{\c e} liczeniu $Q(T)$ ale dow\'od jest zupe{\l}nie elementarny;
zostawiamy go czytelnikom jako, \.ze w Twierdzeniu poni\.zej dowodzimy du\.zo mocniejszy warunek (ale u\.zywaj{\c a}c znanego z literatury
nietrywialnego faktu). \\
(1)(ii) Najpro\'sciej u\.zyc tutaj formu{\l}y produktowej i zobaczy\'c kontrybucj{\c e} ka\.zdego wierzcho{\l}ka do warto\'sci $c_1$:\\
Mamy teraz ${a+b \choose a, b}_q= 1+q+...$ je\'sli $a,b >0$. Z tego wynika, u\.zywaj{\c a}c formuly z dowodu wniosku \ref{Wniosek 2.3}, \.ze
${a_1+a_2+...+a_k \choose a_1,a_2,...a_k}_q = 1+ (k-1)q + ...$ dla $a_1,a_2,...,a_k >0$.  Ko\'ncowy wynik dla $c_1$ wynika z formu{\l}y
produktowej  wniosku \ref{Wniosek 2.3}.\\

(1)(iii) Najpro\'sciej zacz{\c a}\'c od sprawdzenia symetrii $q$-wsp\'o{\l}czynnika Newtona. Mamy:
$$ {a+b \choose a,b}_{q^{-1}} = q^{-ab} {a+b \choose a,b}_q.$$
Dalej u\.zywamy wniosku \ref{Wniosek 2.3} i faktu ze iloczyn wielomian\'ow symetrycznych jest symetryczny.\\
 (2) i (3) Warunki te wynikaj{\c a} bezpo\'srednio ze wzoru produktowego we wniosku \ref{Wniosek 2.3}(ii).
\rule{1.6ex}{1.6ex}\\

\begin{cwiczenie} Znale\'z\'c wz\'or na $c_2$ wielomianu $Q(T)$ dla ka\.zdego drzewa z korzeniem $T$.
\end{cwiczenie}

Trudniejszym do wykazania jest nast{\c e}puj{\c a}cy interesuj{\c a}cy fakt.
\begin{twierdzenie}\label{Theorem 2.6}
Ciag $c_0,c_1,...,c_N$ jes unimodalny, tzn dla pewnego $j$ (tutaj $\lfloor\frac{N}{2}\rfloor$ lub $\lceil \frac{N}{2})\rceil$) mamy
$c_0\leq c_1 \leq ... \leq c_j \geq c_{j+1} \geq \ldots \geq c_N$.

\end{twierdzenie}
{\bf Dow\'od.}
Unimodalno\'s\'c wynika z nietrywialnego faktu, pokazanego przez Sylvestra, \.ze $q$-symbole dwumianowe sa unimodalne, ponadto u\.zywamy
prostszej obserwacji, \.ze
produkt symetrycznych (palindromicznych) dodatnich unimodalnych wielomian\'ow jest unimodalny (see \cite{Sta-1,Win}).
\rule{1.6ex}{1.6ex}\\

Mo\.zna pr\'obowa\'c uog\'olni\'c  Twierdzenie \ref{Theorem 2.6}(iii) odpowiadaj{\c a}c na nast{\c e}puj{\c a}ce pytanie: dla jakich drzew
  ich wielomian $Q(T)$ jest \'sci\'sle unimodalny, to znaczy
$$c_0 < c_1 < ... < c_{\lfloor N/2 \rfloor} = c_{\lceil N/2 \rceil} > \ldots > c_N;$$
(por\'ownaj \cite{Pak-Pan}.).
\section{Komentarze i powi{\c a}zania}

Jak podkre\'sla{\l}em we wst{\c e}pie wielomian ``skubania" jest ciekawy sam w sobie ale nigdy bym go nie skonstruowa{\l} czy odkry{\l} gdybym
nie spostrzeg{\l} jego cienia w moich badaniach w teorii w{\c e}z{\l}\'ow. Dla mnie motywacj{\c a} by{\l}a praca z moim by{\l}ym studentem
Mieczys{\l}awem D{\c a}bkowskim i jego studentem doktoranckim Changsongiem Li dotycz{\c a}ca modu{\l}\'ow motkowych uog\'olnionego (kratowego)
 skrzy\.zowania; rysunek 3.1, \cite{DLP}.
W pracy tej, nie u\.zywamy wielomianu $Q(T)$ jako, \.ze by{\l} on odkryty po napisaniu pracy. B{\c edziemy go u\.zywa\'c w przysz{\l}ych badaniach \cite{D-P}.
W{\c e}z{\l}y motywuj{\c a} te\.z pewne uog\'olnienie wielomianu skubania, przez wyposa\.zenie drzewa w funkcj{\c e} op\'o\'zniaj{\c a}c{\c a},
$f: L(V) \to {\mathcal N}=\{n\in \Z\ | \ n \geq 1\}$,
mowi{\c a}c{\c a} kiedy mo\.zemy u\.zyc li\'scia w formule rekurencyjnej (dajemy tu du\.zo swobody czytelnikom jak zdefiniowa\'c taki wielomian ``skubania"
z funkcj{\c a} op\'o\'zniaj{\c a}c{\c a}).

Zwi{\c a}zek wielomanu ``skubania" z nawiasem Kauffmana splot\'ow jest precyzyjny ale trudny do zwi{\c e}z{\l}ego zapisania. Aby mie\'c pewn{\c a}
ide{\c e} powiem kr\'otko, \.ze dotyczy on badania uog\'olnionego (kratowego) \\
skrzy\.zowania (rysunek 3.1) przy za{\l}o\.zeniu, \.ze
ka\.zde skrzy\.zowanie mo\.zna rozwi{\c a}za\'c relacj{\c a} motkow{\c a} Kauffmana, jak na rysunku 3.2, a ka\.zd{\c a} trywialn{\c a} zamkni{\c e}t{\c a}
sk{\l}adow{\c a} mo\.zna wyeliminowa\'c zast{\c e}puj{\c a}c j{\c a} wielomianem Laurenta $-A^2 - A^{-2}$.

\centerline{\psfig{figure=mncrossingRot.eps,height=3.9cm}}\ \\
\centerline{ Rysunek 3.1; $T_{m\times n}$: $m\times n$ skrzy\.zowanie kratowe}
\ \\ \ \\
\centerline{\psfig{figure=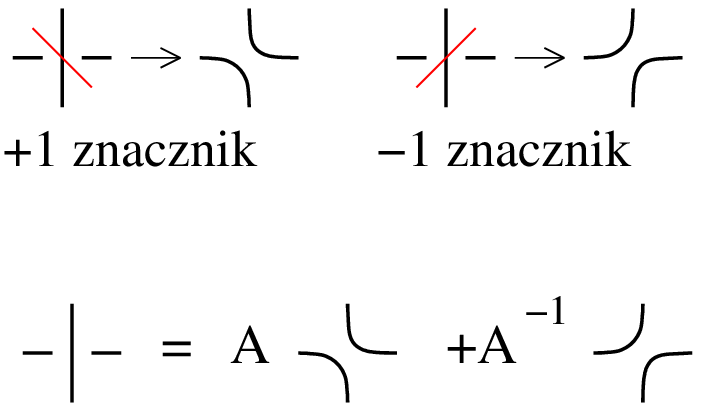,height=4.3cm}}\ \\
\centerline{ Rysunek 3.2; wyg{\l}adzenia Kauffmana skrzy\.zowania i relacja motkowa}
\ \\ \

Wieloman skubania mo\.zna uog\'olni\'c do dowolnych graf\'ow. Jesli graf $G$ ma punkt wyr\'o\.zniony $b$ to niezmiennik $Q(G,b)$ jest po prostu
rodzin{\c a} (wielo-zbiorem) wielomian\'ow skubania wszystkich drzew rozpinajacych graf $G$ z korzeniem $b$ (kr\'otk{\c a} wzmiank{\c e} na temat
tego niezmiennika zamie\'sci{\l}em w \cite{Prz-2}).
Zwi{\c a}zek naszego niezmiennika ze znanymi niezmiennikami graf\'ow wymaga szczeg\'o{\l}owych bada\'n.

Nasz wielomian ma te\.z zwi{\c a}zek z algebr{\c a} homologiczn{\c a}:\\
Niech $\mathcal C$ bedzie kompleksem {\l}a\'ncuchowym to znaczy ci{\c a}giem grup abelowych, $C_n$ i homomorfizm\'ow $\partial_n:C_n \to C_{n-1}$
tak, \.ze $\partial_{n-1} \partial_n = 0$. Na bazie kompleksu {\l}ancuchowego definiujemy grupy homologii przez
$H_n(\mathcal C)= \ker \partial_n/im(\partial_{n+1}$.  Bardzo cz{\c e}sto homomorfizm brzegu, $\partial_n$ jest sum{\c a} alternujac{\c a}
homomorfizm\'ow zwanych funkcjami \'sciany
$\partial_n= \sum_{i=0}^n (-1)^id_i$. Micha{\l} Kapranov zada{\l} pytanie co si{\c e} stanie gdy $(-1)^i$ zostanie zast{\c a}pione przez $q^n$,
to znaczy zdefiniujemy $$\partial^q_n= \sum_{i=0}^n q^id_i.$$ Zauwa\.zy{\l} on, \.ze je\'sli $q$ jest $k$-tym pierwiastkiem z jedno\'sci r\'oznym
od $1$ (tzn. $q^k=1, q\neq 1$) to $k$-ta iteracja $\partial^q$ daje zero ($\partial^q_{n-k+1}...\partial^q_{n-1} \partial^q_n = 0$), \cite{Kapr}.
To, \.ze idea Kapranova jest w jaki\'s spos\'ob zwi{\c a}zana z $q$-wielomianem drzewa z korzeniem wydaje si{\c e} jasne,
ale precyzyjne zwi{\c a}zki wymagaj{\c a} bada\'n (mo\.ze Ty czytelniku tym si{\c e} zajmiesz?).


\ \\ \ \\
Wydzia{\l} Matematyki, \\
George Washington University,\\
oraz University of Maryland College Park,\\
a tak\.ze Uniwersytet Gda\'nski

\end{document}